\documentclass{article}

\usepackage{amsmath}
\usepackage{amssymb}
\usepackage{amsthm}

\newtheorem{thm}{Theorem}

\newtheorem{lem}{Lemma}
\newtheorem{prop}{Proposition}

\title{On moments of gaps between consecutive squarefree numbers}
\author{Tsz Ho Chan}

\begin{document}
\maketitle

\begin{abstract}
Let $s_1, s_2, s_3, \cdots$ be the set of squarefree numbers in ascending order. In this paper, we prove that the following asymptotic on moments of gaps between squarefree numbers
\[
\sum_{s_{k+1} \le x} (s_{k+1} - s_k)^\gamma \sim B(\gamma) x \; \; \mbox{ with some constant} \; \; B(\gamma) > 0
\]
is true for $0 \le \gamma < 3.75$. This improves the previous best range $0 \le \gamma < 3.6875$.

\end{abstract}

\section{Introduction and Main Result}

A positive integer is {\it squarefree} if it is not divisible by the square of any prime number. For example, $6$ is squarefree while $12$ is not. Thus, the first few squarefree numbers, $s_k$, are
\[
s_1 = 1, \; \; s_2 = 2, \; \; s_3 = 3, \; \; s_4 = 5, \; \; s_5 = 6, \; \; s_6 = 7,  \; \; s_7 = 10, \; \; \cdots .
\]
It is well-known that the set of squarefree numbers has asymptotic density $6 / \pi^2 > 1/2$ which implies that there are infinitely many consecutive squarefree number (i.e. $s_{k+1} - s_{k} = 1$ infinitely often). Furthermore, (as pointed out by Huxley \cite{Hu1}) Mirsky \cite{M} proved that
\begin{equation} \label{Mirsky}
\mathop{\sum_{s_{k+1} \le x}}_{s_{k+1} - s_k = h} 1 = \alpha(h) x + O \Bigl(\frac{h x}{\log x \log \log x} \Bigr) \; \; \mbox{ for  any integer } \; \; 1 \le h < \frac{\log x \log \log \log x}{(\log \log x)^2}
\end{equation}
with some constant $\alpha(h)$ independent of $x$ which satisfies
\begin{equation} \label{alpha}
\log \alpha(h) \le - \frac{5}{4} h \log \log h + O(h)
\end{equation}
by \cite[Lemma 1]{Hu1}. Using this, Erd\H{o}s \cite{E} showed the following moment asymptotic on gaps between successive squarefree numbers:
\begin{equation} \label{moment1}
\sum_{s_{k+1} \le x} (s_{k+1} - s_k)^\gamma \sim B(\gamma) x
\end{equation}
for $0 \le \gamma \le 2$. Here
\begin{equation} \label{Bgamma}
B(\gamma) = \sum_{h = 1}^{\infty} h^\gamma \alpha(h)
\end{equation}
which converges by (\ref{alpha}). Hooley \cite{Ho} improved the range of validity of (\ref{moment1}) to $0 \le \gamma \le 3$. Later, Filaseta \cite{F} and Filaseta and Trifonov \cite{FT2} improved it further to $0 \le \gamma < 29/9 = 3.222...$ and $0 \le \gamma < 43/13 = 3.307...$ respectively by differencing method. The current best result is due to Huxley \cite{Hu1} and \cite{Hu2} who showed that (\ref{moment1}) is true for $0 \le \gamma < 11/3 = 3.666...$ and $0 \le \gamma < 59/16 = 3.6875$ respectively by using geometric considerations and results on number of rational points close to a curve. In this paper, we extend the range further to $0 \le \gamma < 3.75$. In particular, we prove the following theorem which gives (\ref{moment1}) for $0 \le \gamma < 3.75$ readily.
\begin{thm} \label{mainthm}
For $0 \le \gamma < 3.75$,
\[
\sum_{x / 2 < s_{k+1} \le x} (s_{k+1} - s_k)^\gamma \sim \frac{1}{2} B(\gamma) x
\]
with $B(\gamma)$ defined by (\ref{Bgamma}).
\end{thm}

\bigskip

The paper is organized as follows. First, we will review Huxley's work \cite{Hu1}. Then, we will refine the treatment of Case 1(a) by using \cite{Hu4} instead of \cite{Hu2}. Finally, we will prove Theorem \ref{mainthm} using a fifth derivative result in \cite{HS2} in place of a fourth derivative result in \cite{HS1} together with finer analysis on different ranges for $H$. Future improvement relies on better treatment of Case 2 in Huxley's work.

\bigskip

{\bf Notation.} We use $\| x \|$ to denote the distance between $x$ and the nearest integer, and $|A|$ to denote the number of elements in the set $A$. The symbols $f(x) = O(g(x))$, $f(x) \ll g(x)$ and $g(x) \gg f(x)$ are equivalent to $|f(x)| \leq C g(x)$ for some constant $C > 0$. The symbol $f(x) \asymp g(x)$ means that $f(x) \ll g(x) \ll f(x)$. Finally, $f(x) \sim g(x)$ means that $\lim_{x \rightarrow \infty} f(x)/g(x) = 1$.

\section{A summary of Huxley's work}

This section highlights some of the key elements in Huxley's work \cite{Hu1}. With $s_{k+1} - s_k = h + 1$, the set of $h$ consecutive integers
\[
\mathcal{S} := \{ s_k+1, s_k + 2, ..., s_{k+1}-1 \}
\]
are all non-squarefree. Suppose $\gamma \ge 3$. In view of Mirsky's asymptotic formula (\ref{Mirsky}) and the current best bound on gaps between consecutive squarefree numbers \cite{FT1}, we can focus on $\frac{1}{2} H_0 \le h \le H_1$ where
\[
H_0 := \Bigl(\frac{\log x}{\log \log x}\Bigr)^{1/(\gamma + 2)} \; \; \mbox{ and } \; \; H_1 := C_0 x^{1/5} \log x
\]
for some constant $C_0 > 0$ (note that $x$ here is the same as $N$ in \cite{Hu1}). Let $H$ run over powers of $2$ and we basically group the gaps $h$ into dyadic intervals $[H, 2H)$. Define
\begin{equation} \label{P}
P_0(H) := \frac{1}{4} H \log H \; \; \mbox{ and } \; \; P_1(H) := H^\gamma \log H.
\end{equation}
Huxley classified the gaps $s_{k+1} - s_k \ge H+1$ into large-prime gaps and small-prime gaps. Large-prime gap means that the set $\mathcal{S}$ contains a number of the form $p^2 q$ with some prime $p \ge P_1(H)$.
Small-prime gap means each subset of $H$ consecutive integers in $\mathcal{S}$ contains at least $H / 4$ numbers of the form $p^2 q$ with prime $P_0(H) \le p < P_1(H)$. A lemma of Erd\H{o}s shows that we must have either large-prime gaps or small-prime gaps. Huxley proved that the contribution from large-prime gaps longer than $H_0$ contributes
\begin{equation} \label{largegap}
O \Bigl( \frac{x}{\log \log x} \Bigr).
\end{equation}
Let
\[
F(H, n) := | \{n < p^2 q \le n+H : p \mbox{ is prime and } p \ge P_0(H) \}|
\]
and
\[
F(H, P, n) := | \{n < p^2 q \le n+H : p \mbox{ is prime and } P \le p < 2P \}|
\]
where $P$ is a power of two. Then a small-prime gap of length greater than $H$ contains a number $n$ with $F(H, n) \ge H/4$ and we also have
\[
\mathop{\sum\nolimits'}_{P} F(H, P , n) \ge \frac{H}{8}
\]
where $\sum\nolimits'$ denotes a sum over powers of two with
\begin{equation} \label{F}
F(H, P, n) \ge \frac{H}{16 \gamma \log H}
\end{equation}
since there are at most $2 \gamma \log H$ different powers of two in the interval $[P_0(H)/2, P_1(H))$. Following \cite{Ho} and \cite{F}, the set of sextuples
\[
\{ (p_1, p_2, p_3, q_1, q_2, q_3)  : P \le p_1, p_2, p_3 < 2P, \; \;
\frac{1}{2}x \le p_3^2 q_3 < p_2^2 q_2 < p_1^2 q_1 \le \min(x, p_3^2 q_3 + H - 1) \Bigr\}
\]
was introduced. It follows from (\ref{F}) that a gap of length greater than $H$ contains a number $n$ such that there are at least
\[
\frac{H^3}{12 \times 2^{12} \gamma^3 \log^3 H}
\]
such sextuples from $(n, n + H]$. With $D := 2^9 (\gamma \log H)^{3/2}$ and $D' := 2 D^2$, we define
\begin{align*}
S(H, P) := \Big| \Bigl\{(p_1, p_2, p_3, q_1, q_2, q_3)&  : P \le p_1, p_2, p_3 < 2P;  (q_1, q_2, q_3) \le D;  (q_1, q_2), (q_2, q_3), (q_3, q_1) \le D'; \\
&\frac{1}{2}x \le p_3^2 q_3 < p_2^2 q_2 < p_1^2 q_1 \le \min(x, p_3^2 q_3 + H - 1) \Bigr\} \Big|
\end{align*}
which counts sextuples with certain greatest common divisor conditions on $q_i$'s. With
\[
T(H, P) := \max_{x/2 \le n \le x - H} \Big| \left\{ 1 \le i \le H : n + i = p^2 q \mbox{ for some prime } P \le p < 2P \right\} \Big|,
\]
Huxley \cite[Lemma 4]{Hu1} established the following lemma.
\begin{lem} \label{simpler}
For $x$ sufficiently large, we have
\[
\sum_{x / 2 < s_{k+1} \le x} (s_{k+1} - s_k)^\gamma = \frac{1}{2} B(\gamma) x + O \Bigl(\frac{x}{\log \log x} \Bigr)
\]
provided that, for each powers of two $H$ and $P$ in the ranges
\[
\frac{1}{2} \Bigl(\frac{\log x}{\log \log x}\Bigr)^{1/(\gamma + 2)} \le H \le C_0 x^{1/5} \log x \; \; \mbox{ and } \; \; \frac{1}{4} H \log H \le P \le H^\gamma \log H,
\]
we have either 
\begin{equation} \label{T}
T(H, P) < \frac{H}{64 \gamma \log H},
\end{equation}
or a bound for $S(H, P)$ in one of the following forms:
\begin{equation} \label{S1}
S(H, P) = O \Bigl(\frac{x}{H^{\gamma - 3} \log^6 H} \Bigr),
\end{equation}
or for some $\eta > 0$ and some $P' \le P$, which may depend on $H$ but not on $P$,
\begin{equation} \label{S2}
S(H, P) = O \Bigl(\frac{x}{H^{\gamma - 3} \log^5 H} \Bigl( \frac{P'}{P} \Bigr)^\eta \Bigr),
\end{equation}
or for some $\eta > 0$ and some $P' \ge P$, which may depend on $H$ but not on $P$,
\begin{equation} \label{S3}
S(H, P) = O \Bigl(\frac{x}{H^{\gamma - 3} \log^5 H} \Bigl(\frac{P}{P'}\Bigr)^\eta \Bigr).
\end{equation}
The implied constants depend on $\gamma$ and $\eta$.
\end{lem}

\bigskip

Next, $S(H, P)$ was bounded in two stages. Firstly, for fixed primes $P \le p_1, p_2, p_3 < 2P$, he studied the number of integer vectors $\vec{q} = (q_1, q_2, q_3)$ that would give a sextuple $(p_1, p_2, p_3, q_1, q_2, q_3)$ counted in $S(H, P)$. Secondly, for $K$ a power of two, he considered $S(H, P, K)$, the number of triplets of primes $p_1, p_2, p_3$ with the number of such vectors $\vec{q}$ lying in the range $K$ to $2K - 1$. Note that
\[
S(H, P) \le \sum_{K} 2K \; S(H, P, K)
\]
where $K$ is over powers of two. Then, Huxley broke it up
\[
S(H, P, K) = S_1(H, P, K) + S_2(H, P, K) + S_3(H, P, K)
\]
into three cases based on some geometric considerations. It was shown that
\begin{equation} \label{S3bound}
\sum_{K} K S_3(H, P ,K) \ll \frac{H^2 x}{P^3 \log^3 P} \; \; \mbox{ which satisfies (\ref{S2}) with $\eta = 3$ for $\gamma < 4$},
\end{equation}
\begin{equation} \label{S2bound1}
\sum_{K} K S_2(H, P, K) \; \; \mbox{ satisfies (\ref{S2}) with $\eta = \frac{3}{2}$ for }  \; P \gg H^{(2 \gamma - 4)/3} \log^{7/3} H,
\end{equation}
and
\begin{equation} \label{S2bound2}
\sum_{K} K S_2(H, P, K) \; \; \mbox{ satisfies (\ref{S3}) with $\eta = \frac{1}{2}$ for } \; P \ll \frac{x}{H^{2 \gamma - 3} \log^5 H}.
\end{equation}
The bottleneck comes from Case 1 which consists of vectors $\vec{q}$ that are multiples of some primitive vector $\vec{r} = (r_1, r_2, r_3)$. It was shown that
\[
S_1(H, P, K) \le \frac{3H}{K} S'(H, P, K)
\]
where $S'(H, P, K)$ is the number of ordered sets of integers $p_1, p_2, r_1, r_2$ with $P \le p_1, p_2 < 2P$ distinct,
\[
\frac{x}{16 K} \le p_1^2 r_1, p_2^2 r_2 \le \frac{x}{K} \; \; \mbox{ and } \; \; 1 \le p_1^2 r_1 - p_2^2 r_2 \le \frac{H}{K}.
\]
By dropping one dimension, Case 1 was further subdivided into Case 1(a) and Case 1(b) depending on a parameter $L$ which counts the number of different integer vectors $(r_1, r_2)$. Hence, with slightly different notation from \cite{Hu1},
\begin{align*}
S'(H, P, K) =& S'_a(H, P, K) + S'_b(H, P, K), \\ 
S_1(H, P, K) =& S_{1a}(H, P, K) + S_{1b}(H, P, K), \\
S_1(H, P) =& S_{1a}(H, P) + S_{1b}(H,P).
\end{align*}
by breaking into (a) and (b) subcases (Note: Huxley used the notation $S_1'(H,P,K)$, $S_2'(H,P,K)$, $S_{11}(H,P)$, $S_{12}(H,P)$ instead of $S_a'(H,P,K)$, $S_b'(H,P,K)$, $S_{1a}(H,P)$, $S_{1b}(H,P)$ here). Also, for $L$ a power of $2$, one has
\[
S'_a(H, P, K) \le \sum_{L} 2L \; S'_a(H, P, K, L) \; \; \mbox{ and } \; \; S'_b(H, P, K) \le \sum_{L} 2L \; S'_b(H, P, K, L)
\]
where $S'_a(H, P, K, L)$ and $S'_b(H, P, K, L)$ are the number of pairs of primes in Case 1(a) or Cae 1(b) respecitvely for which there are between $L$ and $2L - 1$ vectors $(r_1, r_2)$. From the definition of $D$ and $D'$, we have
\[
K \ll \log^{3/2} H \; \; \mbox{ and } \; \; L \ll \log^3 H
\]
in Case 1 by the construction of $S(H,P)$ in \cite{Hu1}. It was shown in \cite[page 200]{Hu1} that
\begin{equation} \label{S1bbound}
\sum_{K} K S_{1b}(H, P ,K) \ll \frac{H^{6/5} x^{3/5}}{\log^{12/5} P} \mbox{ which satisfies (\ref{S1}) for } \; \gamma < 3.8.
\end{equation}

\section{Improvement on Case 1(a)}

Case 1(a) amounts to bounding the number of quadruples $(p_1, p_2, t, u)$ satisfying
\begin{equation} \label{1a}
1 \le |p_1^2 t - p_2^2 u| \le \frac{H}{K L} \; \; \mbox{ with } 1 \le t, u \le \frac{\sqrt{2} x}{K L P^2} \mbox{ and } P \le p_1, p_2 < 2P.
\end{equation}
In \cite{Hu1}, a lemma based on Dirichlet interchange was used to obtain the range $\gamma < 11/3 = 3.666...$. Later in \cite{Hu2}, a theorem on rational points close to a curve was proved to obtain the slightly better range $\gamma < 59/16 = 3.6875$. Here we shall apply a result in \cite{Hu4}.
\begin{prop} \label{closept}
Given real numbers $\lambda > 0$, $C \ge 3/2$, $M \ge 2$ and $Q \ge 2$. Suppose $F(x)$ is a real-valued function $2d+2$ times continuously differentiable on the interval $[0, 1]$ with
\[
|F^{(r)} (x)| \le \lambda C^{r+1}
\]
for $r = 0, 1, 2, ..., 2d+2$,
\[
|D_{r,s}(F(x))| \ge \Bigl( \frac{\lambda}{C^{r+1}} \Bigr)^{s}
\]
for $r = s = d$ and $r = d+1$, $s= 1, 2, ..., d+1$ where
\[
D_{r,s}(F(x)) = \det \Bigl( \frac{F^{(r+i-j)}(x)}{(r + i - j)!} \Bigr)_{s \times s}.
\]
Let $S$ be the set of rational points $(m/n, r/q)$ with $0 \le m \le n$, $1 \le n \le M$, $1 \le q \le Q$, $(m, n) = 1 = (r, q)$ that satisfies $|F(m/n) - r/q| \le \delta$. Let $T = \lambda Q^2$ and $\Delta = \delta Q^2$ with $\Delta < 1/2$ and $T \ge 4$. Then
\[
|S| \ll \bigl( ((1 + \Delta^{1/d} M^2) M^{2d} T)^{1/(2d+1)} + \Delta^{1/(2d+1)} M^2 \bigr) (M T)^\epsilon + (\Delta^{d^2+2d-1} T^{d(d-1)})^{1/(d(d+1)(2d-1))} M^2
\]
for any $\epsilon > 0$. The implied constant may depend on $C$ and $\epsilon$. In the special case of $d = 1$ and $\lambda =1$, one has
\begin{equation} \label{Hux4S}
|S| \ll \bigl( (M Q)^{2/3} + \delta^{1/3} (M Q)^{4/3} \bigr) \log M Q.
\end{equation}
\end{prop}
We also need a simple inequality.
\begin{lem} \label{lem-min}
For any non-negative real numbers $a$, $b$ and $c$,
\begin{equation} \label{mq}
\min(a, b+c) \le \min(a, b) + \min(a, c).
\end{equation}
\end{lem}

\begin{proof}
One simply observes that
\[
\min(a,b) + \min(a,c) = \min(a+a, b+a, a+c, b+c) = \min(a + \min(a,b,c), b+c) \ge \min(a, b+c)
\]
as $a, b, c$ are all non-negative.
\end{proof}

\bigskip

Without loss of generality, we may assume that $p_1 > p_2$ in (\ref{1a}). Hence, it suffices to consider
\[
0 < \Big| \frac{p_1^2}{p_2^2} - \frac{u'}{t'} \Big| \le \frac{H}{K L P^2 t' v} \; \; \mbox{ with } 1 \le t' \le u' \le \frac{\sqrt{2} x}{K L P^2 v}, (u', t') = 1, \mbox{ and } P \le p_2 < p_1 < 2P
\]
after canceling the greatest common divisor $v = (u,t)$. The above is included in
\begin{equation} \label{2dim}
0 < \Big| \sqrt{\frac{u'}{t'}} - \frac{p_1}{p_2} \Big| \le \frac{H}{K L P^2 \sqrt{u' t'} v} \; \; \mbox{ with } 1 \le t' \le u' \le \frac{\sqrt{2} x}{K L P^2 v}, (u', t') = 1, \mbox{ and } P \le p_2 < p_1 < 2P.
\end{equation}
We divide $t'$ into dyadic intervals $[T', 2T')$. Note that as $1 < \frac{p_1}{p_2} \le \frac{2P - 1}{P}$ and $\frac{H}{K L P^2 \sqrt{u' t'} v} \le \frac{4}{P \log H}$, we have $T' \le t' \le u' \le 4 t'$ and $T' \le \sqrt{u' t'} \le 4 T'$. We are going to apply (\ref{Hux4S}) to
\[
F(u) = \left\{ \begin{array}{l} \sqrt{1 + u} \\ \sqrt{2 + u} \\ \sqrt{3 + u} \end{array} \right. \mbox{ with } \lambda = 1, \; \; C= 100, \; \; Q = 2 P, \; \; M = 2T' \le \frac{\sqrt{2} x}{K L P^2 v}, \; \; \delta = \frac{H}{K L P^2 T' v}
\]
depending on whether $0 \le u = \frac{u' - t'}{t'} \le 1$, or $1 \le u = \frac{u' - t'}{t'} \le 2$, or $2 \le u = \frac{u' - t'}{t'} \le 3$. One can check that the derivative conditions in Proposition \ref{closept} are satisfied. The condition $\Delta < 1/2$ is satisfied when
\[
T' > \frac{8 H}{K L v}.
\]
Under this condition, (\ref{Hux4S}) gives an upper bound
\[
O \Bigl( \Bigl(T'^{2/3} P^{2/3} + \frac{H^{1/3} T' P^{2/3}}{K^{1/3} L^{1/3} v^{1/3}} \Bigr) \log x \Bigr)
\]
for the number of rational number solutions to (\ref{2dim}). This, in turn, gives
\begin{equation} \label{Hbd1}
S_a'(H, P, K , L) \ll \Bigl(\frac{x}{K L P^{4/3}} + \frac{H^{1/3} x}{K^{4/3} L^{4/3} P^{4/3}} \Bigr) \log x
\end{equation}
after summing over dyadic intervals $[T', 2T')$ and $v$. When $T' \le \frac{8 H}{K L v}$, one simply applies Lemma 7 in \cite{Hu1} to get the bound
\begin{equation} \label{Hbd2}
S_a'(H, P, K , L) \ll \frac{H^2 \log x}{K^2 L^2}.
\end{equation}
Combining (\ref{Hbd1}) and (\ref{Hbd2}), we have
\[
S_a'(H, P, K , L) \ll \frac{H^2 \log x}{K^2 L^2} + \frac{x \log x}{K L P^{4/3}} + \frac{H^{1/3} x \log x}{K^{4/3} L^{4/3} P^{4/3}}.
\]
Then, following page 201 in \cite{Hu1},
\begin{align*}
S_{1a}(H, P, K) \ll& \frac{H}{K} S'_a (H, P, K) \ll \frac{H}{K} \sum_{L} L \; S_a'(H,P,K,L) \\
\ll& \frac{H}{K} \sum_{L} \min \Bigl( \frac{L P^2}{\log^2 P}, \frac{H^2 \log x}{K^2 L^2} + \frac{x \log x}{K L P^{4/3}} + \frac{H^{1/3} x \log x}{K^{4/3} L^{4/3} P^{4/3}} \Bigr) \\
\ll& \frac{H}{K} \sum_{L} \Bigl[ \frac{H^2 \log x}{K^2 L^2} + \min \Bigl(\frac{L P^2}{\log^2 P}, \frac{x \log x}{K L P^{4/3}} + \frac{H^{1/3} x \log x}{K^{4/3} L^{4/3} P^{4/3}} \Bigr) \Bigr] \\
\ll& \frac{H^3 \log x}{K^3} + \frac{H}{K} \min \Bigl(P^2 \log P, \frac{x \log x}{K P^{4/3}} + \frac{H^{1/3} x \log x}{K^{4/3} P^{4/3}} \Bigr)
\end{align*}
by Lemma \ref{lem-min} and $L \ll \log^3 H$. By Lemma \ref{lem-min},  $\min(a,b) \le a^{\alpha} b^{1 - \alpha}$ with $\alpha = 2/5$ and $K \ll \log^{3/2} H$, we have
\begin{align*}
S_{1a}(H, P) \ll& \sum_{K} K \; S_{1a}(H, P, K) \\
\ll& H^3 \log x + H \sum_{K} \min \Bigl(P^2 \log P, \frac{x \log x}{K P^{4/3}} \Bigr) + H \sum_{K} \min \Bigl(P^2 \log P, \frac{H^{1/3} x \log x}{K^{4/3} P^{4/3}} \Bigr) \\
\ll& H^3 \log x + H \min  \Bigl(P^2 \log^{5/2} P, \frac{x \log x}{P^{4/3}} \Bigr) + H \min \Bigl(P^2 \log^{5/2} P, \frac{H^{1/3} x \log x}{P^{4/3}} \Bigr) \\
\ll& H^3 \log x + H x^{3/5} \log^{8/5} x + H^{6/5} x^{3/5} \log^{8/5} x \ll H^{6/5} x^{3/5} \log^{8/5} x
\end{align*}
as $H \ll x^{1/5} \log x$. This implies
\begin{equation} \label{S1abound}
\sum_{K} K S_{1a}(H, P ,K) \; \mbox{ satisfies (\ref{S1}) for } \; \gamma < 3.8.
\end{equation}

\section{Proof of Theorem \ref{mainthm}}

Instead of a result of Huxley and Sargos \cite{HS1} as used in \cite{Hu1}, we use its improvement \cite[Theorem 5]{HS2}.
\begin{prop} \label{5deriv}
Let $0 < \delta \le 1/4$ and $M \ge 4$. Suppose $f : [M, 2M] \rightarrow \mathbb{R}$ has $r$ continuous derivatives with $|f^{(r)}(u)| \asymp \lambda_r$ for $M \le u \le 2 M$ with some real numbers $\lambda_r > 0$. Define
\[
\mathcal{R}(f, \delta) = | \{ m \in [M, 2M] \cap \mathbb{Z} : \| f(m) \| \le \delta \} |.
\]
Then, for $r \ge 5$,
\[
\mathcal{R}(f, \delta) \ll M \lambda_r^{\frac{2}{r (r+1)}} + M \delta^{\frac{2}{(r - 1)(r - 2)}} + \Bigl( \frac{\delta}{\lambda_{r - 1}} \Bigr)^{\frac{1}{r - 1}} + 1.
\]
\end{prop}

\begin{proof}[Proof of Theorem \ref{mainthm}]
Suppose $3 \le \gamma < 3.8$. Then, (\ref{S3bound}), (\ref{S1bbound}) and (\ref{S1abound}) imply that the contributions from $S_1(H,P,K)$ and $S_3(H,P,K)$ towards $S(H,P)$ satisfy (\ref{S1}) or (\ref{S2}). From (\ref{S2bound1}) and (\ref{S2bound2}), the contribution from $S_2(H,P,K)$ towards $S(H,P)$ satisfies (\ref{S2}) or (\ref{S3}) when
\begin{equation} \label{bridge}
H \le x^{\frac{3}{8 \gamma - 13}} / \log^{22} x 
\end{equation}
as the two intervals $P \gg H^{(2\gamma - 4)/3} \log^{7/3} H$ and $P \ll \frac{x}{H^{2\gamma-3} \log^5 H}$ would overlap.

Now, we try to make (\ref{T}) to hold. For $x/2 \le n \le x$, consider the function $f_n(u) = n / u^2$ with $P \le u < 2P$. With $M = P$ and $\delta = H / P^2$, we have
\begin{equation} \label{Tb}
T(H, P) \ll \max_{x/2 \le n \le x} \mathcal{R}(f_n, \delta) \ll x^{1/15} P^{8/15} + H^{1/6} P^{2/3} + \frac{H^{1/4} P}{x^{1/4}}
\end{equation}
by Proposition \ref{5deriv} with $r = 5$. One can check that the above first upper bound is $< \frac{H}{192 \gamma \log H}$ when $P \le \frac{H^{15/8}}{x^{1/8} \log^2 H}$, the second upper bound is $< \frac{H}{192 \gamma \log H}$ when $P \le \frac{H^{5/4}}{\log^2 H}$, and the third upper bound is $< \frac{H}{192 \gamma \log H}$ when $P \le \frac{H^{3/4} x^{1/4}}{\log^2 H}$. Hence, (\ref{T}) holds unless
\begin{equation} \label{3conditions}
P > \frac{H^{15/8}}{x^{1/8} \log^2 H}, \; \; \frac{H^{5/4}}{\log^2 H} \; \; \mbox{ or } \; \; \frac{H^{3/4} x^{1/4}}{\log^2 H}.
\end{equation}

If $P > \frac{H^{5/4}}{\log^2 H}$, then $P \gg H^{(2 \gamma - 4)/3} \log^{7/3} H$ for $\gamma < 3.875$. This would imply $S(H,P)$ satisfies (\ref{S1}), (\ref{S2}) or (\ref{S3}) by (\ref{S3bound}), (\ref{S1bbound}), (\ref{S1abound}) and (\ref{S2bound1}).

\bigskip

If $P > \frac{H^{3/4} x^{1/4}}{\log^2 H}$, then $P \gg H^{(2 \gamma - 4)/3} \log^{7/3} H$ for $\gamma < 5$. This would imply $S(H,P)$ satisfies (\ref{S1}), (\ref{S2}) or (\ref{S3}) by (\ref{S3bound}), (\ref{S1bbound}), (\ref{S1abound}) and (\ref{S2bound1}).

\bigskip

Finally, if $P > \frac{H^{15/8}}{x^{1/8} \log^2 H}$, then $P \gg H^{(2 \gamma - 4)/3} \log^{7/3} H$ when $H > x^{\frac{3}{77 - 16 \gamma}} \log^7 x$ as $\gamma < 3.8$. So, when $H > x^{\frac{3}{77 - 16 \gamma}} \log^7 x$, $S(H,P)$ satisfies (\ref{S1}), (\ref{S2}) or (\ref{S3}) by (\ref{S3bound}), (\ref{S1bbound}), (\ref{S1abound}) and (\ref{S2bound1}). When $H \le x^{\frac{3}{77 - 16 \gamma}} \log^7 x$, $H$ satisfies (\ref{bridge}) as long as $\gamma < 3.75$. Thus, when $\gamma < 3.75$, $S(H, P)$ satisfies (\ref{S1}), (\ref{S2}) or (\ref{S3}) regardless the size of $H$.

\bigskip

Therefore, when $3 \le \gamma < 3.75$, one of the conditions in Lemma \ref{simpler} is satisfied and we have Theorem \ref{mainthm}.
\end{proof}

Mathematics Department \\
Kennesaw State University \\
Marietta, GA 30060 \\
tchan4@kennesaw.edu

\end{document}